\documentclass[12pt,reqno]{amsart}

\usepackage{preamble_paper}
\usepackage{preamble_general}
\usepackage{preamble_math}

\begin{document}

\title{A note on the shortest law for the symmetric group}

\author{Adrian Beker}
\address{University of Zagreb, Faculty of Science, Department of Mathematics, Zagreb, Croatia}
\email{Adrian.Beker@math.hr}

\author{Luka Mili\'{c}evi\'{c}}
\address{Mathematical Institute of the Serbian Academy of Sciences and Arts, Belgrade, Serbia}
\email{Luka.Milicevic@turing.mi.sanu.ac.rs}

\author{Rudi Mrazovi\'{c}}
\address{University of Zagreb, Faculty of Science, Department of Mathematics, Zagreb, Croatia}
\email{Rudi.Mrazovic@math.hr}

\begin{abstract}
    Let $\alpha(n)$ denote the length of the shortest non-trivial two-variable law for the symmetric group $S_n$.
    Buskin's quantitative subgroup-separability argument gives the classical lower bound $\alpha(n)\geq 2n-O(1)$.
    In this short note we give an improvement by proving that $\alpha(n)\geq \frac52 n-O(1)$.
\end{abstract}

\maketitle

\section{Introduction}

Let $F_2=\langle a,b\rangle$ be the free group on two generators, whose elements are considered as reduced words over the alphabet $\cA=\{a,a^{-1},b,b^{-1}\}$.
A non-trivial word $w\in F_2$ is a \emph{law} for a group $G$ if
\[
  w(g,h)=1
  \qquad\text{for every }g,h\in G.
\]
We write $|w|$ for the length of a word $w$ and define 
\[
    \alpha(n) \vcentcolon= \min\bigl\{|w| : w \text{ is a law for the symmetric group $S_n$}\bigr\}.
\]

The problem of estimating $\alpha(n)$ is closely related to quantitative subgroup separability in free groups.
Buskin proved \cite{Buskin} that every non-trivial word $w\in F_2$ can be omitted by a subgroup of index at most $|w|/2+2$; this implies $\alpha(n)\geq 2n-3$.
In the opposite direction, Kozma and Thom \cite{KozmaThom},
using the quasipolynomial diameter bound of Helfgott and Seress \cite{HelfgottSeress} for Cayley graphs of $S_n$,
proved
\[
    \alpha(n)
    \leq
    \exp\bigl(O((\log n)^4\log\log n)\bigr).
\]
Thus even the superlinearity of $\alpha(n)$ is an open problem.

As mentioned above, 
Buskin phrased his result in terms of the divisibility index,
which is defined for $1\neq w\in F_2$ as
\[
  D_{F_2}(w)
  \vcentcolon=\min\bigl\{[F_2:H]: H\leq F_2,\ w\notin H\bigr\}.
\]
The standard coset-action argument gives
\begin{equation}
    \label{eq:law-divisibility}
    w\text{ is a law for }S_n
    \quad\Longleftrightarrow\quad
    D_{F_2}(w)>n.
\end{equation}
Indeed, a subgroup $H$ of index at most $n$ yields a coset action in which $w$ moves the coset $H$ precisely when $w\notin H$. 
Conversely, if a pair $(\sigma, \tau)$ of permutations in $S_n$ witnesses that $w$ is not a law, considering the homomorphism $F_2 \to S_n$ mapping $v \mapsto v(\sigma, \tau)$, the preimage of the stabilizer of a point moved by $w(\sigma, \tau)$ has index at most $n$ and does not contain $w$.

Our main result is the following very modest improvement of Buskin's bound. Equivalently,
by \eqref{eq:law-divisibility}, we have a new lower bound for the length of the shortest law for $S_n$.

\begin{theorem}
    \label{t:main-divisibility}
    For every non-trivial word $w\in F_2$,
    \[
        D_{F_2}(w)\leq \frac25|w|+O(1).
    \]
    Equivalently, for every $n \in \N$,
    \[
        \alpha(n)\geq \frac52 n-O(1).
    \]
\end{theorem}

\begin{remark}
    The constants $2/5$ and $5/2$ appearing in Theorem~\ref{t:main-divisibility} can almost certainly be improved.
    They come from the particular local constructions used in Lemma~\ref{l:long-terminal},
    together with the finite verification in Lemma~\ref{l:finite-gadgets}.
    A more extensive search, and a refinement of the local gadgets allowed in
    Lemma~\ref{l:long-terminal}, 
    would likely improve these constants.
    For this reason we have not tried to optimize the additive $O(1)$ terms in the statements,
    although the proof is completely effective and explicit
    constants could be extracted from it.
    On the other hand, improving the lower bound on $\alpha(n)$ to a superlinear one would require genuinely new ideas,
    rather than merely a more careful optimization of the constructions used here.
\end{remark}

Theorem~\ref{t:main-divisibility} will be a simple consequence of the corresponding statement for labelled permutation digraphs, 
which we now introduce. For a letter $s\in\cA$, we define its label to be the underlying generator in $\{a,b\}$
(i.e.\ $a$, $a^{-1}$ have label $a$, and $b$, $b^{-1}$ have label $b$).

\begin{definition}
    A \textit{labelled permutation digraph} $D$ consists of a finite vertex set and, for each $C\in\{a,b\}$, exactly one outgoing and one incoming directed edge labelled $C$ at every vertex. 
    Thus the $C$-edges form the directed cycles of a permutation of the vertex set (see Figure \ref{f:lab-perm-digraph}).
    Loops are allowed.
\end{definition}

\begin{figure}[ht]
    \centering
    \begin{tikzpicture}[->,>=stealth',auto,node distance=3cm,thick,
    	main node/.style={circle,draw,minimum size=6mm}]
    	
    	\node[main node] (1) at (90:3cm) {1};
    	\node[main node] (2) at (141:3cm) {2};
    	\node[main node] (3) at (193:3cm) {3};
    	\node[main node] (4) at (244:3cm) {4};
    	\node[main node] (5) at (296:3cm) {5};
    	\node[main node] (6) at (347:3cm) {6};
    	\node[main node] (7) at (399:3cm) {7};
    
    	\draw (1) to (2);
    	\draw (2) to (4);
    	\draw (4) to (1);
    	\path (3) edge[loop left] ();
    	\path (6) edge[loop right] ();
    	\draw (5) to (7);
    	\draw (7) to[bend left=18] (5);
    
    	\draw[dashed] (1) to (7);
    	\draw[dashed] (7) to (4);
    	\draw[dashed] (4) to (3);
    	\draw[dashed] (3) to (1);
    	\draw[dashed] (2) to[bend left=18] (6);
    	\draw[dashed] (6) to[bend left=18] (2);    
    	\path[dashed] (5) edge[loop below] ();
    \end{tikzpicture}
    \caption{The labelled permutation digraph on seven vertices corresponding to
    $(1\,2\,4)\,(3)\,(6)\,(5\,7)$ (solid edges) 
    and
    $(1\,7\,4\,3)\,(2\,6)\,(5)$ (dashed edges). Dashed and solid edges correspond to two possible labels $C \in \{a,b\}$.}
    \label{f:lab-perm-digraph}
\end{figure}
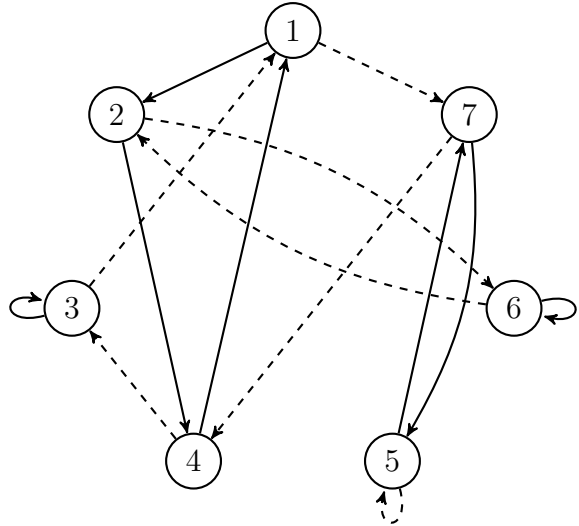

Given a word $w=s_1\dots s_{\ell}$, a labelled digraph $D$, and a vertex $v$, the \textit{$(w,v)$-walk} is the walk that starts at $v$ and successively reads the letters $s_1,\ldots,s_{\ell}$:
if $s_i=C$, it follows the outgoing $C$-edge, while if $s_i=C^{-1}$, it follows the incoming $C$-edge, where $C\in \{a,b\}$.
When the starting vertex is clear or irrelevant, we simply call it a \textit{$w$-walk}.

Theorem~\ref{t:main-divisibility} will follow easily from the following statement, which is formulated in terms of labelled permutation digraphs.

\begin{theorem}
    \label{t:detector}
    For every non-empty reduced word $w$, there exists a non-closed $w$-walk on a labelled permutation digraph with at most $2|w|/5+O(1)$ vertices.
\end{theorem}

Although our general strategy is close in spirit to Buskin's approach, 
especially in some of the manipulations used to construct labelled permutation digraphs, 
we believe that our proof contains a few new ingredients which result in a better bound.
We give a short overview of these ideas.
For a word segment we construct a non-closed walk and record two quantities: the number of vertices used (called the cost), and the word length plus the number of edges traversed exactly once (called the value).
The latter quantity is monotone under extension of the word. 
Carefully chosen segments have value-to-cost ratio at least $5/2$, 
and a gluing operation makes these two quantities additive while preserving the goodness of the walk. 
Long powers are handled by cycles whose lengths are least non-divisors of their exponents.
The remaining bounded collection of short terminal-power words is settled by exhaustive verification.
Further intuition for these constructions is given below, in the remarks and discussion between the proofs.

\section{Constructions of labelled permutation digraphs}

Let $\eta(D,w,v)$ be the number of directed labelled edges traversed exactly once by the $(w,v)$-walk on $D$, and define its \emph{value} as
\[
    \Phi(D,w,v)
    \vcentcolon=
    |w|+\eta(D,w,v).
\]

We seek to construct non-closed walks by concatenating short walks into longer ones. In order to carry out this construction, we have to strengthen the condition of being non-closed. This is made precise by the following definition.

\begin{definition}
    Suppose that the $(w,v)$-walk on $D$ ends at $u$.
    The triple $(D,w,v)$ is \emph{good} if
    \begin{enumerate}[label=(\roman*)]
        \item $u\neq v$;
        \item all traversed edges incident with $v$ have the same label;
        \item all traversed edges incident with $u$ have the same label.
    \end{enumerate}
    For a good triple $(D,w,v)$,
    we define its \emph{cost} and \emph{profile} as
    \[
        \cost(D,w,v) \vcentcolon= |D|-1,
        \qquad
        P(D,w,v)
        \vcentcolon=
        \bigl(\cost(D,w,v),\Phi(D,w,v)\bigr),
    \]
    and its \emph{efficiency} as\footnote{Note that the cost of a good triple is always strictly positive.}
    \[
      \rho(D,w,v)
      \vcentcolon=
      \frac{\Phi(D,w,v)}{\cost(D,w,v)}.
    \]
\end{definition}

Our first lemma describes the gluing operation for good triples. In fact, the notion of goodness was designed so that the gluing operation can be successfully carried out. In the proof, it will be useful to identify permutations with the corresponding directed graphs. Given a permutation $\sigma$ of a finite set $X$ and an element $x \in X$, by ``$\sigma$ with $x$ deleted'' we will mean the directed graph obtained by removing the edges $(\sigma^{-1}(x),x)$, $(x,\sigma(x))$, and adding the edge $(\sigma^{-1}(x),\sigma(x))$ unless $\sigma(x) = x$. Thus, $x$ is an isolated vertex in this graph, and its removal gives a permutation of $X\setminus\{x\}$.

\begin{lemma}[Gluing]
    \label{l:gluing}
    Let $w_1,\ldots,w_k$ be non-empty reduced words such that, for each $i<k$, the label of the last letter of $w_i$ differs from that of the first letter of $w_{i+1}$. Let $w = w_1\dots w_k$ be their concatenation, and suppose $(D_i,w_i,v_i)$ is a good triple for each $i$.
    Then there is a good triple $(D,w,v)$ satisfying
    \[
        P(D,w,v)
        =
        \sum_{i=1}^k P(D_i,w_i,v_i).
    \]
\end{lemma}

\begin{proof}
    By induction, it is enough to glue two good triples, i.e.\ we may assume that $k = 2$.
    Let the $(w_1,v_1)$-walk end at $u_1$, and suppose without loss of generality that the last label of $w_1$ is $b$ and the first label of $w_2$ is $a$.

    Take disjoint copies of $D_1$ and $D_2$ and identify $u_1$ with $v_2$. On the resulting vertex set, define the $a$-permutation by using the $a$-permutation of $D_1$ with $u_1$ deleted,
    together with the full $a$-permutation of $D_2$. 
    Define the $b$-permutation symmetrically:
    retain the full $b$-permutation of $D_1$ and use the $b$-permutation of $D_2$ with $v_2$ deleted. Let $D$ be the resulting labelled permutation digraph and set $v = v_1$. 
    The goodness conditions ensure that the $(w_1,v_1)$-walk on $D_1$ never uses an $a$-edge incident with $u_1$, and the $(w_2,v_2)$-walk on $D_2$ never uses a $b$-edge incident with $v_2$.
    Hence, if we view them as walks on $D$, these two walks remain unchanged and the $(w,v)$-walk is precisely their concatenation.
    
    It is now clear that $\eta(D,w,v) = \eta(D_1,w_1,v_1) + \eta(D_2,w_2,v_2)$, and also that $|D|-1=(|D_1|-1)+(|D_2|-1)$.
    Moreover, the initial and terminal goodness properties of $(D,w,v)$ follow from those of the two original triples. Thus, the proof of the lemma is complete.
\end{proof}

The following lemma highlights the significance of edges that are traversed once and is a key ingredient in our approach.

\begin{lemma}[Edge splitting]
    \label{l:edge-splitting}
    Let $w$ be a non-empty reduced word, $D$ a labelled permutation digraph and $v$ a vertex of $D$. 
    If the $(w,v)$-walk traverses some edge exactly once, then either that walk is non-closed, 
    or there is a labelled permutation digraph on $|D|+1$ vertices with a non-closed $w$-walk.
\end{lemma}

\begin{proof}
    Assume the walk is closed and let $e$ be an edge which is traversed exactly once, say the directed $C$-edge $(x,y)$. 
    Add a new vertex $z$ and replace the $C$-edge $(x,y)$ with $C$-edges $(x,z)$ and $(z,y)$.
    Additionally, give $z$ a loop labelled with the other generator. Let $D'$ be the labelled permutation digraph obtained by making these modifications.
    
    Write $w=pC^{\varepsilon}r$ for some $\varepsilon\in\{-1,1\}$, where the displayed occurrence of the label $C$ corresponds to the unique step that traverses $e$. Assume that $\varepsilon = 1$; the case $\varepsilon = -1$ is similar. The key observation is that the $(p,v)$-walk and $(r^{-1},v)$-walk on $D$ do not use the edge $e$ and hence coincide with the corresponding walks on $D'$. But this means that on $D'$, the $(pC,v)$-walk ends at $z$, whereas the $(r^{-1},v)$-walk ends at $y$. Since these two vertices are different, the $(w,v)$-walk on $D'$ is not closed, as desired.
\end{proof}

We next describe a simple way to construct good triples with a reasonably small labelled permutation digraph.

\begin{lemma}[Naive construction]
    \label{l:naive}
    For every non-empty reduced word $w$, 
    there is a good triple with profile $(|w|,2|w|)$.
\end{lemma}

\begin{proof}
    Consider first $w$ of the form $C^d$ for some $C \in \{a,b\}$ and $d\neq 0$. 
    Use a directed $C$-cycle of length $|d|+1$, 
    attach a loop of the other label at each vertex,
    and start at any vertex. 
    The walk is non-closed, traverses $|d|$ distinct $C$-edges exactly once, and no edges labelled with the other generator. Hence, the resulting triple is good, and its profile is $(|d|,2|d|)$. 
    For a general word $w$, split it into maximal powers of $a$ and $b$ and apply Lemma~\ref{l:gluing}.
\end{proof}

Finally, the following elementary observation is the reason for adding $\eta$ to the word length in the definition of the quantity $\Phi$.

\begin{lemma}[Prefix monotonicity]
    \label{l:monotonicity}
    Let $w'$ be a prefix of the word $w$. Then $\Phi(D,w',v) \leq \Phi(D,w,v)$. In particular, if $\Phi(D,w',v) > |w|$, then $\eta(D,w,v) \geq 1$.
\end{lemma}

\begin{proof}
    For the first statement, it suffices to consider the case when $|w'| = |w|-1$. The last step of the $(w,v)$-walk traverses either a previously unused edge, or an edge previously traversed exactly once, 
    or an edge previously traversed at least twice.
    Accordingly, when passing from $w'$ to $w$, $\eta$ changes by $1,-1$, or $0$. 
    Since the word length increases by one, the desired conclusion follows. The second statement follows from the first since if $\eta(D,w,v) = 0$, then $\Phi(D,w,v) = |w|$.
\end{proof}

Let us now reinterpret the proof of Buskin's $|w|/2+3$ bound for Theorem \ref{t:detector} in the new language introduced in this section\footnote{Strictly speaking, Buskin obtains the slightly better bound $|w|/2+2$, but this is not important for the present discussion.}.
Given a word $w$ of length $\ell$, our task is to produce a labelled permutation digraph on at most $\ell/2+3$ vertices with a non-closed $w$-walk.
Let $w'$ be the prefix of $w$ of length $\ell' = \lfloor\ell/2\rfloor+1$. 
By Lemma \ref{l:naive}, there is a good triple $(D,w',v)$ such that $P(D,w',v) = (\ell',2\ell')$. In particular, since $2\ell' > \ell$, by Lemma \ref{l:monotonicity} we have $\eta(D,w,v) \geq 1$. Hence, by Lemma \ref{l:edge-splitting}, either the $(w,v)$-walk on $D$ is not closed or there is a labelled digraph on $\ell'+2$ vertices with a non-closed $w$-walk, as desired.

To improve on this, the idea is to establish a version of Lemma \ref{l:naive} which gives better results in certain special cases. 
To this end, observe that the construction in Lemma \ref{l:naive} is rather wasteful when $w$ contains many long segments consisting of the same letter. 
Indeed, if $w = a^d$ for some $d \in \N$, 
we could have equally well taken an $a$-cycle of length $\nd(d)$ (see \eqref{eq:nd} for a definition) with a $b$-loop attached to each vertex. 
Note that we trivially have $\nd(d) \leq d+1$, so this cannot perform worse than the construction in Lemma \ref{l:naive}.
Moreover, by the prime number theorem, $\nd(d) \ll \log d$, so if $d$ is large, we obtain an enormous saving in the number of vertices of $D$.
In fact, it is not hard to see that in this way we get $\rho(D,w,v) > 2$ as soon as $d > 2$. 
However, this observation is not sufficient on its own, as if $d \in \{1,2\}$, then $\nd(d) = d+1$, so the construction based on the least non-divisor does not yield any improvement whatsoever. In the next sections, we explain how to improve the construction further despite this obstacle.

\section{Efficient terminal-power gadgets}

For a positive integer $d$, define its least non-divisor by
\begin{equation}
    \label{eq:nd}
    \nd(d)
    \vcentcolon=
    \min\{q\in\N : q\nmid d\}.
\end{equation}

\begin{lemma}[Least non-divisor construction]
    \label{l:nondivisor}
    Let $d>2$ and $C\in\{a,b\}$. 
    Then the word $w = C^{\pm d}$ has a good triple with profile
    \[
        \bigl(\nd(d)-1,\max\{2\nd(d),d\}\bigr).
    \]
\end{lemma}

\begin{proof}
    Let $q=\nd(d)$.
    Form a labelled permutation digraph $D$ by taking a directed $C$-cycle of length $q$ and attaching loops of the other label; let $v$ be an arbitrary vertex od $D$.
    It is clear the $(w,v)$-walk traverses only edges labelled with $C$, and since $q\nmid d$, it is not closed. 
    Write $d=tq+r$ with $t \in \N_0$ and $1\leq r\leq q-1$. 
    Since $d>2$, we have $q\leq d-1$, so $t\geq1$.
    If $t=1$, then $2q-d=q-r$ edges are traversed exactly once, and hence $\Phi(D,w,v)=2q$. 
    If $t\geq2$, then every edge is traversed at least twice and so $\Phi(D,w,v)=d$.
\end{proof}

Thus, the problematic words are those that consist of short segments of $a$'s and $b$'s. 
To deal with such words, the idea is to break them up into smaller chunks,
for which the above construction can be improved. 
As a rough indication that this may be possible, 
consider a word $w = a^{d_1} b^{e_1}\dots a^{d_k} b^{e_k}$ and suppose there exists $i < k$ such that $d_i + d_{i+1} \neq 0$. 
Then one can essentially merge the cycles corresponding to $a^{d_i}$ and $a^{d_{i+1}}$ into a single cycle, 
and replace the cycle corresponding to $b^{e_i}$ by a single fixed point. 
This reduces the cost and the value by roughly the same amount, which results in a larger efficiency.
It remains to deal with those words such that $d_{i+1} = -d_i$ and $e_{i+1} = -e_i$ for all $i < k$. 
But such words are essentially periodic with period of the form $a^d b^e a^{-d} b^{-e}$, 
and can be dealt with by a separate argument.
Even though these arguments could be carried out analytically, 
it turns out that a computational approach is both cleaner and delivers quantitatively superior results.

\begin{lemma}[Long terminal powers]
    \label{l:long-terminal}
    Let
    \[
        w=w'C^{\varepsilon d},
        \qquad C\in\{a,b\},\quad \varepsilon\in\{-1,1\},
    \]
    be a reduced word, where $w'$ is empty or ends in the other generator, $|w'|\leq4$, and $|w|\geq 5$.
    Then there is a good triple $(D,w,v)$ such that
    \[
        \rho(D,w,v)\geq \frac52
        \qquad\text{and}\qquad
        \eta(D,w,v)\ll 1.
    \]
\end{lemma}

\begin{proof}
    For $|w|\leq 10$, this can be verified by a brute-force computer search; see Appendix \ref{s:code}.
    So we assume $|w|\geq 11$, and hence $d\geq 7$.
    Use the naive construction for $w'$ (Lemma \ref{l:naive}) and the least-non-divisor construction (Lemma \ref{l:nondivisor}) for $C^{\varepsilon d}$, and glue them together using Lemma \ref{l:gluing}. The resulting triple has profile
    \[
        \bigl( 
            |w'|+\nd(d)-1,\,2|w'|+\max\{2\nd(d),d\}
        \bigr).
    \]
    It is a straightforward matter to check that the minimum of
    \begin{equation}\label{eq:ratio}
        \frac{2\ell + \max\{2\nd(d), d\}}{\ell+\nd(d)-1}
    \end{equation}
    over all $d \geq 7$ and $0\leq \ell \leq 4$ is $5/2$. Indeed, if $d \geq 12$, then $\frac{d}{3}(\frac{d}{3}-1) \geq d$, so since $(\nd(d)-1)(\nd(d)-2)$ is either $0$ or a divisor of $d$, we must have $\nd(d)-1 \leq \frac{d}{3}$. Thus, the difference between the ratio \eqref{eq:ratio} and $5/2$ is at least
    \[
        \frac{2\ell+d}{\ell+d/3} - \frac{5}{2} = \frac{d/3-\ell}{\ell+d/3} \geq 0,
    \]
    with equality if and only if $(d,\ell) = (12, 4)$. The remaining cases $7 \leq d \leq 11$ can be easily verified by hand. We thus obtain a good triple with efficiency $\rho$ at least $5/2$, and the required control on $\eta$ follows 
    since for sufficiently large $d$ we have $\nd(d) < d/2$.
\end{proof}

\section{Proof of the main theorem}

\begin{proof}[Proof of Theorem \ref{t:detector}]
    Let $w = s_1\ldots s_{\ell}$, where $\ell = |w|$ is the length of $w$. Write $w$ in syllable form
    \[
        w=C_1^{f_1}C_2^{f_2}\dots C_m^{f_m},
    \]
    where $C_i\in\{a,b\}$, consecutive $C_i$ are different, and every $f_i$ is non-zero.
    Greedily partition these syllables into consecutive chunks $c_1,\ldots,c_r$ by repeating the following: starting at the first unused syllable, keep concatenating syllables until the accumulated length is at least $5$; if fewer than $5$ letters remain, take them as the final chunk.
    
    Every chunk except possibly the last has length at least $5$.
    Moreover, in each such chunk the part preceding its final syllable has length at most $4$.
    The same holds for the last chunk whenever its length is at least $5$.
    Thus every prefix of a chunk having length at least $5$ satisfies the hypotheses of Lemma~\ref{l:long-terminal}.
    
    For each chunk prefix $z$, choose a good triple as follows.
    If $1\leq|z|\leq4$, use Lemma~\ref{l:naive} to obtain a good triple with profile $(|z|,2|z|)$.
    If $|z|\geq5$, use Lemma~\ref{l:long-terminal} instead.
    
    For a global prefix $w' = s_1\ldots s_{\ell'}$, write it in the form $w' = c_1\dots c_{j-1}w''$, where $w''$ is a prefix of $c_j$. Consider the chosen good triples for $c_1,\ldots,c_{j-1}$ and $w''$.
    Note that all of these triples have efficiency $\rho$ at least $5/2$,
    except possibly the last one if $|w''| \leq 4$.
    Glue these triples using Lemma~\ref{l:gluing} to obtain a good triple $(D_{\ell'},w',v_{\ell'})$ such that
    \[
        |D_{\ell'}| \leq \frac{2}{5} \Phi(D_{\ell'},w',v_{\ell'}) + O(1).
    \]
    Call $w'$ \emph{suitable} if $\Phi(D_{\ell'},w',v_{\ell'}) > \ell$. 
    
    Suppose first that a suitable prefix $w'$ exists, and consider the one whose length $\ell'$ is minimal.
    By minimality, we then have $\Phi(D_{\ell'-1},s_1\ldots s_{\ell'-1},v_{\ell'-1}) \leq \ell$, and by construction, $\Phi(D_{\ell'},w',v_{\ell'})$ and $\Phi(D_{\ell'-1},s_1\ldots s_{\ell'-1},v_{\ell'-1})$ differ by at most $O(1)$. Hence, we have $\Phi(D_{\ell'},w',v_{\ell'}) \leq \ell + O(1)$.
    Moreover, Lemma \ref{l:monotonicity} implies that $\eta(D_{\ell'},w,v_{\ell'}) \geq 1$ and hence Lemma \ref{l:edge-splitting} implies that
    either the $(w,v_{\ell'})$-walk on $D_{\ell'}$ is not closed or there is a digraph with $|D_{\ell'}|+1$ vertices and a non-closed $w$-walk. 
    We thus obtain a digraph with at most $2\ell/5 + O(1)$ vertices and a non-closed $w$-walk, as desired.
    
    If a suitable prefix $w'$ does not exist, then in particular $w$ is not suitable, meaning that $\Phi(D_{\ell},w,v_{\ell}) \leq \ell$. 
    In this case, it readily follows that $|D_{\ell}| \leq 2\ell/5 + O(1)$ 
    and the $(w,v_{\ell})$-walk on $D_{\ell}$ is not closed, so we are again done.
\end{proof}

\begin{proof}[Proof of Theorem~\ref{t:main-divisibility}]
    Given any non-trivial word $w \in F_2$, apply Theorem~\ref{t:detector} to obtain a labelled permutation digraph $D$ with $n \leq 2|w|/5+O(1)$ vertices and a vertex $v$ such that the $(w,v)$-walk is not closed. By considering the corresponding $a$-permutation and $b$-permutation, we obtain elements $\sigma, \tau \in S_n$ such that $w(\sigma,\tau)$ has a non-fixed point. In particular, it follows that $w$ is not a law for $S_n$, thus proving the theorem.
\end{proof}

\appendix
\section{A verifier for Lemma \ref{l:long-terminal}}
\label{s:code}

The proof of Lemma \ref{l:long-terminal} relies on the following finite check.

\begin{lemma}
    \label{l:finite-gadgets}
    Let $w$ be as in Lemma \ref{l:long-terminal}, and assume in addition that $|w| \leq 10$.
    Then there is a good triple $(D,w,v)$ with
    \[
        \rho(D,w,v)\geq \frac52.
    \]
    Moreover, $D$ may be chosen to have at most five vertices.
\end{lemma}

\begin{proof}
    This is verified by exhaustive search.
    A labelled permutation digraph on $q$ vertices $0,1,\dots,q-1$ is exactly an ordered pair in $S_q^2$, and the start vertex may be fixed as $0$.
    The verifier in \cite{py-code} generates all $1944$ reduced words satisfying the hypotheses, enumerates every ordered pair of permutations for $q=2,3,4,5$, simulates the corresponding walk, tests the three goodness conditions, and checks whether $2\Phi(D,w,0)\geq5(q-1)$.
    Every word receives a witness.
    The numbers of words first receiving a witness on $2,3,4,5$ vertices are respectively $12, 216, 1092, 624$.
    These sum to $1944$, completing the finite verification.
\end{proof}

{\small
\subsection*{Acknowledgements and AI tool disclosure}
We would like to thank Sean Eberhard for useful discussions.

AB and RM are supported by the Croatian Science Foundation under the project no.\ HRZZ-IP-2022-10-5116 (FANAP) and by the European Union – NextGenerationEU through the National Recovery and Resilience Plan 2021-2026 Institutional grant of University of Zagreb Faculty of Science (IK IA 1.1.3.\ Impact4Math). LM was supported by the Ministry of Science, Technological Development and Innovation of the Republic of Serbia through the Mathematical Institute of the Serbian Academy of Sciences and Arts.

ChatGPT Plus 5.5 was used to assist with 
routine typesetting,
grammar checking,
phrasing,
identifying technical corrections such as tracking constants,
and drafting Python code \cite{py-code} for the finite verification.
Apart from these uses, the text of this paper was human-written.
}

\printbibliography

@article {KozmaThom,
    AUTHOR = {Kozma, Gady and Thom, Andreas},
     TITLE = {Divisibility and laws in finite simple groups},
   JOURNAL = {Math. Ann.},
  FJOURNAL = {Mathematische Annalen},
    VOLUME = {364},
      YEAR = {2016},
    NUMBER = {1-2},
     PAGES = {79--95},
      ISSN = {0025-5831,1432-1807},
   MRCLASS = {20D06 (20B30 20E05 20P05)},
  MRNUMBER = {3451381},
MRREVIEWER = {Timothy\ C.\ Burness},
       DOI = {10.1007/s00208-015-1201-4},
       URL = {https://www.doi.org/10.1007/s00208-015-1201-4},
}

@article {Buskin,
    AUTHOR = {Buskin, Nikolai V.},
     TITLE = {Economical separability in free groups},
   JOURNAL = {Sib. Math. J.},
  FJOURNAL = {Siberian Mathematical Journal},
    VOLUME = {50},
      YEAR = {2009},
    NUMBER = {4},
     PAGES = {603--608},
      ISSN = {0037-4466,1573-9260},
   MRCLASS = {20E05 (20E07 20E26)},
  MRNUMBER = {2583614},
       DOI = {10.1007/s11202-009-0067-7},
       URL = {https://www.doi.org/10.1007/s11202-009-0067-7},
}

@article {HelfgottSeress,
    AUTHOR = {Helfgott, Harald A. and Seress, \'Akos},
     TITLE = {On the diameter of permutation groups},
   JOURNAL = {Ann. of Math. (2)},
  FJOURNAL = {Annals of Mathematics. Second Series},
    VOLUME = {179},
      YEAR = {2014},
    NUMBER = {2},
     PAGES = {611--658},
      ISSN = {0003-486X,1939-8980},
   MRCLASS = {20B30 (20D06 20F05 20F69)},
  MRNUMBER = {3152942},
MRREVIEWER = {Martin\ W.\ Liebeck},
       DOI = {10.4007/annals.2014.179.2.4},
       URL = {https://www.doi.org/10.4007/annals.2014.179.2.4},
}

@misc{py-code,
    AUTHOR = {Beker, Adrian and Mili{\'c}evi{\'c}, Luka and Mrazovi{\'c}, Rudi},
     TITLE = {Verification code for improved lower bounds on laws of symmetric groups},
      YEAR = {2026},
HOWPUBLISHED = {Python code},
       DOI = {10.5281/zenodo.20802552},
}

\end{document}